\documentclass[12pt]{article}
\usepackage{amsmath,amssymb,amsthm}
\usepackage{amsfonts, amsmath}
\usepackage[english]{babel}
\newcommand{\R}{\mathbb R}

\newcommand{\ep}{\varepsilon}

\newtheorem{theorem}{Theorem}[section]
\newtheorem{lemma}{Lemma}[section]
\theoremstyle{remark}
\newtheorem{remark}{Remark}[section]
\newtheorem{proposition}{Proposition}[section]
\newtheorem{example}{Example}[section]

\numberwithin{equation}{section}

\begin{document}
\title{\bf Multiple solutions to a Robin problem with indefinite weight and asymmetric reaction}
\author{
\bf Giuseppina D'Agu\`{\i}\\
\small{DICIEAMA,
University of Messina,}\\
\small{98166 Messina, Italy}\\
\small{\it E-mail: dagui@unime.it}\\
\mbox{}\\
\bf Salvatore A. Marano\thanks{Corresponding author}\\
\small{Department of Mathematics and Computer Sciences,
University of Catania,}\\
\small{Viale A. Doria 6, 95125 Catania, Italy}\\
\small{\it E-mail: marano@dmi.unict.it}\\
\mbox{}\\
\bf Nikolaos S. Papageorgiou\\
\small{Department of Mathematics,
National Technical University of Athens,}\\
\small{Zografou Campus, Athens 15780, Greece}\\
\small{\it E-mail:npapg@math.ntua.gr}
}
\date{}
\maketitle
\begin{abstract}
The existence of two nontrivial smooth solutions to a semilinear Robin problem with indefinite unbounded potential and asymmetric nonlinearity $f$ is established. Both crossing and resonance are allowed. A third nonzero solution exists provided $f$ is $C^1$. Proofs exploit variational methods, truncation techniques, and Morse theory.
\end{abstract}
\vspace{2ex}

\vspace{2ex}

\section{Introduction} \label{S1}
Let $\Omega$ be a bounded domain in $\R^N$ having a smooth boundary $\partial\Omega$, let $a\in L^s(\Omega)$ for appropriate $s\geq 1$, and let $f:\Omega\times\R\to\R$ be a Carath\'eodory function. The semilinear elliptic equation with indefinite unbounded potential
$$-\Delta u+a(x)u=f(x,u)\quad\mbox{in}\quad\Omega$$
has by now been widely investigated under Dirichlet or Neumann boundary conditions; see \cite{KyPa,PaPa} and \cite{PaRa1,PaSm}, respectively, besides the references given there. If $a(x)\equiv 0$ then the case of asymmetric nonlinearities $f$, meaning that $t\mapsto f(x,t)t^{-1}$ crosses at least the principal eigenvalue of the relevant differential operator as $t$ goes from $-\infty$ to $+\infty$, was also studied; cf. \cite{dePa,dePaPr,ReRu}. From a technical point of view, the Fu{\v c}ik spectrum is often exploited \cite{ACCG}, which entails that the limits $\displaystyle{\lim_{t\to\pm\infty}}f(x,t)t^{-1}$ do exist.

This work treats equations having both difficulties under Robin boundary conditions. Hence, for $a(x)$ bounded only from above, $s>N$, and $\beta\in W^{1,\infty}(\partial\Omega)$ nonnegative, we consider the problem
\begin{equation}\label{prob}
\left\{
\begin{array}{ll}
-\Delta u+a(x)u=f(x,u) & \mbox{in }\Omega,\\
\displaystyle{\frac{\partial u}{\partial n}}+\beta(x) u=0 & \mbox{on } \partial\Omega,\\
\end{array}
\right.
\end{equation}
where $\frac{\partial u}{\partial n}:=\nabla u\cdot n$, with $n(x)$ being the outward unit normal vector to $\partial\Omega$ at its point $x$. As usual,  $u\in H^1(\Omega)$ is called a (weak) solution of \eqref{prob} provided
$$\int_\Omega\nabla u\cdot\nabla v\, dx+\int_{\partial\Omega}\beta uv\, d\sigma+\int_\Omega auv\, dx =\int_\Omega f(x,u)v\, dx\quad\forall\, v\in H^1(\Omega).$$
Our assumptions on the reaction $f$ at infinity are essentially the following.
\begin{itemize}
\item There exists $k\geq 2$ such that $\displaystyle{\hat\lambda_k\leq\liminf_{t\to-\infty}\frac{f(x,t)}{t}\leq \limsup_{t\to-\infty}\frac{f(x,t)}{t}\leq\hat\lambda_{k+1}}$,
\item $\displaystyle{\limsup_{t\to+\infty}\frac{f(x,t)}{t}\leq\hat\lambda_1}$, and $\displaystyle{\lim_{t\to+\infty}\left[f(x,t)t-2\int_0^t f(x,\tau)d\tau\right]=+\infty}$
\end{itemize}
uniformly in $x\in\Omega$. Here, $\hat\lambda_n$ denotes the $n^{\rm th}$-eigenvalue of the problem 
\begin{equation}\label{eigen}
 -\Delta u+a(x)u=\lambda u\quad\mbox{in}\quad\Omega,\quad\frac{\partial u}{\partial n}+\beta(x) u=0\quad\mbox{on}\quad \partial\Omega.
\end{equation}
It should be noted that a possible interaction (resonance)  with eigenvalues is allowed. If an additional condition on the behavior of $t\mapsto f(x,t)t^{-1}$ as $t\to 0$ holds then we obtain at least two nontrivial $C^1$-solutions to \eqref{prob}, one of which is positive; see Theorems \ref{thmone}--\ref{thmthree} for precise statements. As an example, Theorem \ref{thmone} applies when
$$f(x,t):=\left\{
\begin{array}{ll}
bt & \mbox{if } t\leq 1,\\
\hat\lambda_1t-\sqrt{t}+(b-\hat\lambda_1+1)t^{-1} & \mbox{otherwise},\\
\end{array}
\right.$$
with $\hat\lambda_k\leq b\leq\hat\lambda_{k+1}$ and $k>2$ large enough, or
$$f(x,t):=\left\{
\begin{array}{ll}
b(t+1)-c & \mbox{if } t<-1,\\
ct & \mbox{if } |t|\leq 1,\\
\hat\lambda_1(t-1)+c & \mbox{otherwise},\\
\end{array}
\right.$$
where $c>\hat\lambda_2$. Let us point out that, unlike previous results, the nonlinearities treated by Theorem \ref{thmthree} turn out to be concave near zero. Finally, Theorem \ref{thmfour} gives a third nontrivial $C^1$-solution once
$$f(x,\cdot)\in C^1(\R)\quad\mbox{and}\quad\sup_{t\in\R}|f'_t(\cdot,t)|\in L^\infty(\Omega).$$
Our arguments are patterned after those of \cite{MaPaRL} (cf. also \cite{MaMoPa}) where, however, the Dirichlet problem is investigated, $a(x)\equiv 0$, but the $p$-Laplace operator appears. Moreover, the hypotheses on $f$ made there do not permit resonance at any eigenvalue. The approach we adopt exploits variational and truncation techniques, as well as results from Morse theory. Regularity of solutions basically arises from \cite{W}.
\section{Preliminaries}\label{S2}
Let $(X,\Vert\cdot\Vert)$ be a real Banach space. Given a set $V\subseteq X$, write $\overline{V}$ for the closure of $V$, $\partial V$ for the boundary of $V$, and ${\rm int}(V)$ for the interior of $V$. If $x\in X$ and $\delta>0$ then $B_\delta(x):=\{ z\in X:\;\Vert z-x\Vert<\delta\}$ while $B_\delta:=B_\delta(0)$. The symbol $(X^*,\Vert\cdot \Vert_{X^*})$ denotes the dual space of $X$, $\langle\cdot,\cdot\rangle$ indicates the duality pairing between $X$ and $X^*$, while $x_n\to x$ (respectively, $x_n\rightharpoonup x$) in $X$ means `the sequence $\{x_n\}$ converges strongly (respectively, weakly) in $X$'. We say that $\Phi:X\to\mathbb{R}$ is coercive iff
$$\lim_{\Vert x\Vert\to+\infty}\Phi(x)=+\infty.$$
$\Phi$ is called weakly sequentially lower semi-continuous when $x_n\rightharpoonup x$ in $X$ implies
$$\Phi(x)\leq \displaystyle{\liminf_{n\to\infty}}\Phi(x_n).$$
Let $\Phi\in C^1(X)$. The classical Cerami compactness condition for $\Phi$ reads as follows.
\begin{itemize}
\item[$({\rm C})$] {\it Every sequence $\{x_n\}\subseteq X$ such that $\{\Phi(x_n)\}$ is bounded and
$$\lim_{n\to+\infty}(1+\Vert x_n\Vert)\Vert \Phi'(x_n) \Vert_{X^*}=0$$
has a convergent subsequence.}
\end{itemize}
Define, provided $c\in\mathbb{R}$,
$$\Phi_c:=\{x\in X:\; \Phi(x)\geq c\},\quad \Phi^c:=\{x\in X:\; \Phi(x)\leq c\},$$
$$K(\Phi):=\{x\in X:\,\Phi'(x)=0\},\quad K_c(\Phi):=K(\Phi)\cap\Phi^{-1}(c).$$
Given a topological pair $(A,B)$ fulfilling $B\subset A\subseteq X$, the symbol $H_q(A,B)$, $q\in\mathbb{N}_0$, indicates the
${\rm q}^{\rm th}$-relative singular homology group of $(A,B)$ with integer coefficients. If $x_0\in K_c(\Phi)$ is an isolated point of $K(\Phi)$ then
$$C_q(\Phi,x_0):=H_q(\Phi^c\cap V,\Phi^c\cap V\setminus\{x_0\}),\quad q\in\mathbb{N}_0,$$
are the critical groups of $\Phi$ at $x_0$. Here, $V$ stands for any neighborhood of $x_0$ such that $K(\Phi)\cap\Phi^c\cap V=\{x_0\}$. By excision, this definition does not depend on the choice of $V$. Suppose $\Phi$ satisfies Condition $({\rm C})$, $\Phi|_{K(\Phi)}$ is bounded below, and $c<\displaystyle{\inf_{x\in K(\Phi)}}\Phi(x)$. Put
$$C_q(\Phi,\infty):=H_q(X,\Phi^c),\quad q\in\mathbb{N}_0.$$
The Second Deformation Lemma \cite[Theorem 5.1.33]{GaPaNA} implies that this definition does not depend on the choice of $c$. If $K(\Phi)$ is finite, then setting
$$M(t,x):=\sum_{q=0}^{+\infty}{\rm rank}\, C_q(\Phi,x)t^q\, ,\quad
P(t,\infty):=\sum_{q=0}^{+\infty}{\rm rank}\, C_q(\Phi,\infty)t^q\quad
\forall\, (t,x)\in\mathbb{R}\times K(\Phi)\, ,$$
the Morse relation below holds:
\begin{equation}\label{morse}
\sum_{x\in K(\Phi)}M(t,x)=P(t,\infty)+(1+t)Q(t)\, ,
\end{equation}
where $Q(t)$ denotes a formal series with nonnegative integer coefficients; see for instance \cite[Theorem 6.62]{MoMoPa}. 
\begin{proposition}\label{abstract}
Let $h\in C^1([0,1]\times X)$. Assume that:
\begin{itemize}
\item[$({\rm i}_1)$] $h$ maps bounded sets into bounded sets.
\item[$({\rm i}_2)$] $h(0,\cdot)$ and $h(1,\cdot)$ satisfy Condition (C).
\item[$({\rm i}_3)$] $t\mapsto h'_t(t,x)$ is locally Lipschitz continuous. Moreover, there exist $\alpha>0$, $p\in(1,+\infty)$ such that $|h'_t(t,x)|\leq\alpha\Vert x\Vert^p$ in $[0,1]\times X$.
\item[$({\rm i}_4)$] $x\mapsto h'_x(t,x)$ is locally Lipschitz continuous and with appropriate $a,\delta>0$ one has
$$h(t,x)\leq a\;\implies\; (1+\Vert x\Vert)|h'_x(t,x)|\geq\delta\Vert x\Vert^p.$$
\end{itemize}
Then $C_q(h(0,\cdot),\infty)=C_q(h(1,\cdot),\infty)$ for all $q\in\mathbb{N}_0$.
\end{proposition}
This result represents a slight generalization of \cite[Proposition 3.2]{LiSu}. Therefore, we omit the proof.

Now, let $X$ be a Hilbert space, let $x\in K(\Phi)$, and let $\Phi$ be $C^2$ in a neighborhood of $x$. If $\Phi''(x)$ turns out to be invertible, then $x$ is called non-degenerate. The Morse index $d$ of $x$ is the supremum of the dimensions of the vector subspaces of $X$ on which $\Phi''(x)$ turns out to be negative definite. When $x$ is non-degenerate and with Morse index $d$ one has
\begin{equation}\label{kd}
C_q(\Phi,x)=\delta_{q,d}\mathbb{Z}\, ,\quad q\in\mathbb{N}_0\, .
\end{equation}
The monographs \cite{MW, MoMoPa} represent general references on the subject.

Throughout this paper, $\Omega$ denotes a bounded domain of the real euclidean $N$-space $(\mathbb{R}^N,|\cdot|)$ whose boundary is $C^2$. On $\partial\Omega$ we will employ the $(N-1)$-dimensional Hausdorff measure $\sigma$. The Trace Theorem \cite[Theorem 2.79]{CLM} ensures that there exists a unique completely continuous linear operator $\gamma:H^1(\Omega)\to L^2(\partial\Omega)$ such that
$$\gamma(u)=u|_{\partial\Omega}\quad\forall\, u\in C^1(\overline{\Omega}),\quad{\ker}(\gamma)=H^1_0(\Omega).$$
To simplify notation, we let $u$ in place of $\gamma(u)$ when no confusion can arise. The symbol $\Vert\cdot\Vert_q$ with $q\geq 1$ indicates the usual norm of $L^q(\Omega)$ and
$$\Vert u\Vert:=\left(\Vert\nabla u\Vert_2^2+\Vert u\Vert_2^2\right)^{1/2},\quad u\in H^1(\Omega),$$
$$C_+:=\{u\in C^0(\overline{\Omega}): u(x)\geq 0\quad\forall\, x\in\overline{\Omega}\}.$$
Write $2^*$ for the critical exponent of the Sobolev embedding $H^1(\Omega)\subseteq L^q(\Omega)$. Recall that $2^*=2N/(N-2)$ if $2<N$, $2^*=+\infty$ otherwise, and the embedding is compact whenever $1\leq q<2^*$. Moreover,
$${\rm int}(C_+)=\{ u\in C_+: u(x)>0\quad\forall\, x\in\overline{\Omega}\}.$$
Given $t\in\mathbb{R}$, $u,v:\Omega\to\mathbb{R}$, and $f:\Omega\times\mathbb{R}\to\mathbb{R}$, define
$$t^\pm:=\max\{\pm t,0\},\quad u^\pm (x):=u(x)^\pm ,\quad N_f(u)(x):=f(x,u(x)).$$
$u\leq v$ signifies $u(x)\leq v(x)$ for almost every $x\in\Omega$. The meaning of $u<v$ etc. is analogous.
\begin{remark}\label{equiv}
If $u\in H^1(\Omega)$, $w\in L^2(\Omega)$, and $\beta\in L^\infty(\partial\Omega)$ then the condition
$$\int_\Omega\nabla u(x)\cdot\nabla v(x)dx+\int_{\partial\Omega}\beta(x)u(x)v(x)d\sigma=\int_\Omega w(x)v(x)dx,\quad v\in H^1(\Omega),$$
is equivalent to  
$$-\Delta u=w\quad\mbox{a.e. in}\quad\Omega,\quad\frac{\partial u}{\partial n}+\beta(x)u=0\quad\mbox{on}\quad
\partial\Omega.$$
This easily comes out from the nonlinear Green's identity \cite[Theorem 2.4.54]{GaPaNA}; see for instance the proof of \cite[Proposition 3]{PaRa}.
\end{remark}
We shall employ some facts about the spectrum of the operator $u\mapsto-\Delta u+a(x)u$ in $H^1(\Omega)$ with homogeneous Robin boundary conditions. So, consider the eigenvalue problem \eqref{eigen} where, from now on,
\begin{equation}\label{abeta}
a\in L^s(\Omega)\;\mbox{ for some $s>N$, }\; a^+\in L^\infty(\Omega),\quad\beta\in W^{1,\infty}(\partial\Omega),\;\mbox{ and }\;\beta\geq 0.
\end{equation}
Define
\begin{equation}\label{defE}
{\cal E} (u):=\Vert\nabla u\Vert_2^2+\int_\Omega a(x)u(x)^2dx+\int_{\partial\Omega}\beta(x)u(x)^2d\sigma\quad\forall\, u\in H^1(\Omega).
\end{equation}
\begin{lemma}\label{hatab}
There exist $\hat a, \hat b>0$ such that
$${\cal E}(u)+\hat a\Vert u\Vert_2^2\geq \hat b\Vert u\Vert^2\quad\forall\, u\in H^1(\Omega).$$
\end{lemma}
\begin{proof}
If the conclusion was false, we could construct a sequence $\{u_n\}\subseteq H^1(\Omega)$ fulfilling
\begin{equation}\label{limone}
{\cal E}(u_n)+n\Vert u_n\Vert_2^2<\frac{1}{n}\Vert u_n\Vert^2,\quad n\in\mathbb{N}.
\end{equation}
Set $v_n:=\Vert u_n\Vert^{-1}u_n$. Since 
\begin{equation}\label{propone}
\Vert v_n\Vert=1\quad\forall\, n\in\mathbb{N},
\end{equation}
we may assume that
\begin{equation}\label{conv}
v_n\rightharpoonup v\;\mbox{ in }\; H^1(\Omega),\quad v_n\to v\;\mbox{ in }\; L^2(\Omega),\quad\mbox{and}\quad v_n\to v\;\mbox{ in }\; L^2(\partial\Omega).
\end{equation}
Therefore,
\begin{equation}\label{limtwo}
{\cal E}(v)\leq\liminf_{n\to+\infty}{\cal E}(v_n).
\end{equation}
From \eqref{limone}--\eqref{limtwo} it follows $v=0$ as well as $n\Vert v_n\Vert^2_2\to 0$, which implies
\begin{equation}\label{limthree}
\lim_{n\to+\infty}\Vert v_n\Vert=0.
\end{equation}
In fact, on account of \eqref{limone},
$$0={\cal E}(0)\leq\liminf_{n\to+\infty}{\cal E}(v_n)\leq\limsup_{n\to+\infty}{\cal E}(v_n)\leq\lim_{n\to+\infty}
\left(\frac{1}{n}-n\Vert v_n\Vert^2_2\right)=0$$
and, by \eqref{conv},
$$0=\lim_{n\to+\infty}{\cal E}(v_n)=\lim_{n\to+\infty}\Vert\nabla v_n\Vert^2_2.$$
However, \eqref{limthree} contradicts \eqref{propone}.
\end{proof}
Thanks to the above lemma, letting
$$(u,v):=\int_\Omega \nabla u\cdot\nabla v dx+\int_\Omega (a(x)+\hat a)uv dx+\int_{\partial\Omega}\beta(x)uvd\sigma \quad\forall\, u,v\in H^1(\Omega)$$
produces a scalar product on $H^1(\Omega)$ equivalent to the usual one. Further, given $u\in L^2(\Omega)$, there exists a unique $\tilde u\in H^1(\Omega)$ such that
$$(\tilde u,v)=\int_\Omega u(x)v(x)dx,\quad v\in H^1(\Omega).$$ 
Let $K:L^2(\Omega)\to H^1(\Omega)$ be defined by
$$K(u):=\tilde u\quad\mbox{for every}\quad u\in L^2(\Omega)$$
 and let $i:H^1(\Omega)\to L^2(\Omega)$ be the embedding map. Obviously, $K\circ i:H^1(\Omega)\to H^1(\Omega)$ is linear, compact, self-adjoint, while
$$(K\circ i(u),v)=\int_\Omega u(x)v(x)dx\quad\forall\, u,v\in H^1(\Omega).$$
Consequently,
$$(K\circ i(u),u)=\Vert u\Vert_2^2,\quad u\in H^1(\Omega).$$
Theorem 3.1.57 in \cite{GaPaNA} ensures that $K\circ i$ possesses a decreasing sequence $\{\mu_n\}$ of positive eigenvalues such that $\mu_n\to 0$. Then 
$$\hat\lambda_n:=\frac{1}{\mu_n}-\hat a,\quad n\in\mathbb{N},$$
represent the eigenvalues of \eqref{eigen} and there exists a corresponding sequence $\{\hat u_n\}\subseteq H^1(\Omega)$ of eigenfunctions, which turns out to be an orthonormal basis of $H^1(\Omega)$. For each $n\in\mathbb{N}$, denote by $E(\hat\lambda_n)$ the eigenspace associated with $\hat\lambda_n$. It is known that:
\begin{itemize}
\item[$({\rm p}_1)$] \textit{$E(\hat\lambda_n)$ is finite dimensional.}
\item[$({\rm p}_2)$] \textit{If $u$ lies in $E(\hat\lambda_n)$ and vanishes on a set of positive Lebesgue measure, then $u=0$.}
\item[$({\rm p}_3)$] \textit{$E(\hat\lambda_n)\subseteq C^1(\overline{\Omega})$.}
\item[$({\rm p}_4)$] \textit{$H^1(\Omega)=\overline{\oplus_{n=1}^\infty E(\hat\lambda_n)}$. Moreover,
\begin{equation}\label{lone}
\hat\lambda_1=\inf\left\{\frac{{\cal E}(u)}{\Vert u\Vert_2^2}:u\in H^1(\Omega),\, u\neq 0\right\},
\end{equation}
\begin{equation*}
\hat\lambda_n=\inf\left\{\frac{{\cal E}(u)}{\Vert u\Vert_2^2}:u\in\hat H_n,\, u\neq 0\right\}
=\sup\left\{\frac{{\cal E}(u)}{\Vert u\Vert_2^2}:u\in\bar H_n,\, u\neq 0\right\},\; n\geq 2,
\end{equation*}
where $\hat H_n:=\oplus_{i=n}^\infty E(\hat\lambda_i)$ and $\bar H_n:=\oplus_{i=1}^n E(\hat\lambda_i)$.}
\item[$({\rm p}_5)$] \textit{Elements of $E(\hat\lambda_1)$ do not change sign and $\hat\lambda_1$ is simple.}
\item[$({\rm p}_6)$] \textit{There exists an $L^2$-normalized eigenfunction $\hat u_1\in{\rm int}(C_+)$ associated with $\hat\lambda_1$.} 
\item[$({\rm p}_7)$] \textit{Each $\hat\lambda_n$ with $n\geq 2$ possesses a nodal eigenfunction.}
\end{itemize}
In particular, $({\rm p}_2)$ comes out from \cite[Proposition 3]{deFG}, the regularity results of \cite[Section 5]{W} imply $({\rm p}_3)$, while $({\rm p}_5)$ is easily verified through Picone's identity \cite[p. 255]{MoMoPa} besides \eqref{lone}. The same holds true for $({\rm p}_7)$; see, e.g., \cite[Section 9.3]{MoMoPa}. Finally, Theorems 2.5.2 and 5.5.1 in \cite{PuSe} basically yield $({\rm p}_6)$.

The next characterization of $\hat\lambda_2$ will be used later. Its proof is analogous to that of \cite[Proposition 5]{PaRa}.
\begin{itemize}
\item[$({\rm p}_8)$] \textit{Write $M:=\{ u\in H^1(\Omega):\; \Vert u\Vert_2=1\}$ as well as
$$\Gamma_1:=\{\gamma\in C^0([-1,1],M):\gamma(-1)=-\hat u_1,\;\gamma(1)=\hat u_1\}.$$
Then
$$\hat\lambda_2=\inf_{\gamma\in\Gamma_1}\max_{t\in [-1,1]}{\cal E}(\gamma(t)) .$$
}
\end{itemize}
A simple argument, based on orthogonality, $({\rm p}_2)$, and $({\rm p}_4)$, gives the next result.
\begin{lemma}\label{splitting}
Let $n\in\mathbb{N}$ and let $\theta\in L^\infty(\Omega)\setminus\{\hat\lambda_n\}$ satisfy $\theta\geq\hat\lambda_n$. Then there exists a constant $\bar c>0$ such that
$${\cal E}(u)-\int_\Omega\theta(x) u(x)^2dx\leq-\bar c\Vert u\Vert^2\quad\forall\, u\in\bar H_n\, .$$
Let $n\in\mathbb{N}$ and let $\theta\in L^\infty(\Omega)\setminus\{\hat\lambda_n\}$ satisfy $\theta\leq\hat\lambda_n$. Then there exists a constant $\hat c>0$ such that
$${\cal E}(u)-\int_\Omega\theta(x) u(x)^2dx\geq\hat c\Vert u\Vert^2\quad\forall\, u\in\hat H_n\, .$$
\end{lemma}
Finally, consider the weighted eigenvalue problem
\begin{equation}\label{pep}
-\Delta u+a(x)u=\lambda\alpha(x) u\quad\mbox{in}\quad\Omega,\quad \frac{\partial u}{\partial n}+\beta(x)u=0\quad
\mbox{on}\quad\partial\Omega,
\end{equation}
 where $\alpha\in L^\infty(\Omega)\setminus\{0\}$ and $\alpha\geq 0$. Arguing as before produces an increasing sequence $\{\hat\lambda_n(\alpha)\}$ of eigenvalues for \eqref{pep}, which enjoys similar properties. In particular, via the analogue of $({\rm p}_2)$ we achieve the following (cf. \cite[Proposition 1]{deFG}):
\begin{itemize}
\item[$({\rm p}_9)$] \textit{If $\alpha_1,\alpha_2\in L^\infty(\Omega)\setminus\{0\}$, $0\leq\alpha_1\leq\alpha_2$, and $\alpha_1\neq\alpha_2$ then $\hat\lambda_n(\alpha_2)<\hat\lambda_n(\alpha_1)$ for all $n\in\mathbb{N}$.}
\end{itemize}
\section{Existence results}\label{S3}
To avoid unnecessary technicalities, `for every $x\in\Omega$' will take the place of `for almost every $x\in\Omega$' and the variable $x$ will be omitted when no confusion can arise. Define
$$n_0:=\inf\{ n\in\mathbb{N}:\hat\lambda_n\geq 0\}.$$
Let $f:\Omega\times\mathbb{R}\to\mathbb{R}$ be a Carath\'eodory function such that $f(\cdot,0)=0$ and let
\begin{equation}\label{defF}
F(x,\xi):=\int_0^\xi f(x,t)dt\, ,\quad (x,\xi)\in\Omega\times\mathbb{R}.
\end{equation}
We will posit the following assumptions.
\begin{itemize}
\item[$({\rm f}_1)$] \textit{There exists $a_0\in L^\infty(\Omega)$ such that}
$$|f(x,t)|\leq a_0(x)(1+|t|)\quad\forall\, (x,t)\in\Omega\times\R.$$
\item[$({\rm f}_2)$] \textit{$\displaystyle{\limsup_{t\to+\infty}\frac{f(x,t)}{t}}\leq\hat\lambda_1$ and $\displaystyle{\lim_{t\to+\infty}}[f(x,t)t-2F(x,t)]=+\infty$ uniformly in $x\in\Omega$.}
\item[$({\rm f}_3)$] \textit{For some $k\geq\max\{n_0,2\}$ one has
$$\hat\lambda_k\leq\liminf_{t\to-\infty}\frac{f(x,t)}{t}\leq\limsup_{t\to-\infty}\frac{f(x,t)}{t}\leq\hat\lambda_{k+1}$$
uniformly with respect to $x\in\Omega$ and
\begin{equation}\label{rminus}
f(x,t)t-2F(x,t)\geq 0\quad\forall\, (x,t)\in\Omega\times\R^-_0.
\end{equation}
}
\item[$({\rm f}_4)$] \textit{There exist $a_1,a_2\in L^\infty(\Omega)$ such that $\hat\lambda_2<a_1\leq a_2$ and
$$a_1(x)\leq\liminf_{t\to 0}\frac{f(x,t)}{t}\leq\limsup_{t\to 0}\frac{f(x,t)}{t}\leq a_2(x)$$
uniformly in $x\in\Omega$.}
\end{itemize}
We start by pointing out the next auxiliary results.
\begin{lemma}\label{auxone}
Let $({\rm f}_2)$ be satisfied. Then
$$\lim_{t\to+\infty}[\hat\lambda_1t^2-2F(x,t)]=+\infty$$
 uniformly with respect to $x\in\Omega$.
\end{lemma}
\begin{proof}
Given any $K>0$, one can find $\delta_K>0$ such that  $f(x,t)t-2F(x,t)>K$ for all $(x,t)\in\Omega\times [\delta_K,+\infty)$. Hence,
$$\frac{d}{dt}\left(\frac{F(x,t)}{t^2}\right)>\frac{K}{t^3}\quad\mbox{provided}\quad t\geq\delta_K$$
and, a fortiori,
$$\frac{F(x,\xi)}{\xi^2}-\frac{F(x,t)}{t^2}>-\frac{K}{2}\left(\frac{1}{\xi^2}-\frac{1}{t^2}\right)$$
whenever $\xi\geq t\geq\delta_K$. Since
$$\limsup_{\xi\to+\infty}\frac{2F(x,\xi)}{\xi^2}\leq\hat\lambda_1\quad\mbox{uniformly in $x\in\Omega$,}$$
the above inequality produces $\hat\lambda_1 t^2-2F(x,t)\geq K$ for every $(x,t)\in\Omega\times [\delta_K,+\infty)$. As $K$ was arbitrary, the conclusion follows.
\end{proof}
\begin{lemma}\label{auxtwo}
If \eqref{abeta}, $({\rm f}_1)$, and $({\rm f}_4)$ hold, then every nontrivial solution $u_0\geq 0$ of \eqref{prob} belongs to ${\rm int}(C_+)$.
\end{lemma}
\begin{proof}
Using $({\rm f}_1)$ and $({\rm f}_4)$ we get $c_0>0$ such that $|f(x,t)|\leq c_0|t|$ in $\Omega\times\mathbb{R}$. Therefore, the function $b:\Omega\to\R$ defined by
$$b(x):=\frac{f(x,u_0(x))}{u_0(x)}\quad\mbox{if $u_0(x)\neq 0$,}\quad b(x):=0\quad\mbox{otherwise,}$$
is essentially bounded. Since $u_0$ turns out to be a weak solution of the problem
$$-\Delta u=[b(x)-a(x)]u\quad\mbox{in $\Omega$,}\quad\frac{\partial u}{\partial n}+\beta(x)u=0\quad\mbox{on $\partial\Omega$,}$$
where, because of \eqref{abeta}, $b-a\in L^s(\Omega)$ for some $s>N$,  Lemma 5.1 in \cite{W} and  the Moser iteration technique yield  $u_0\in L^\infty(\Omega)$. Through  \cite[Lemma 5.2]{W} we achieve $u_0\in C^{1,\alpha}(\overline{\Omega})$. So, in particular, $u_0\in C_+\setminus\{0\}$. Finally, from
$$\Delta u_0(x)\leq\left(\Vert a^+\Vert_\infty+\Vert b\Vert_\infty\right) u_0(x)\quad\mbox{for every $x\in\Omega$}$$
and the Boundary Point Lemma \cite[p. 120]{PuSe} it follows $u_0\in{\rm int}(C_+)$, as desired.
\end{proof}
To shorten notation, write $X:=H^1(\Omega)$. The energy functional $\varphi:X\to\R$ stemming from Problem \eqref{prob} is
\begin{equation}\label{defphi}
\varphi(u):=\frac{1}{2}{\cal E}(u)-\int_\Omega F(x,u(x))\, dx,\quad u\in X,
\end{equation}
with ${\cal E}$ and $F$ given by \eqref{defE} and (\ref{defF}), respectively. One clearly has $\varphi\in C^1(X)$. Moreover,
\begin{proposition}\label{cerami}
Under \eqref{abeta} and $({\rm f}_1)$--$({\rm f}_3)$, the functional $\varphi$ satisfies Condition (C). 
\end{proposition}
\begin{proof}
Let $\{u_n\}$ be a sequence in $X$ be such that
\begin{equation}\label{bounded}
|\varphi(u_n)|\leq c_1\quad\forall\, n\in\mathbb{N},
\end{equation}
\begin{equation}\label{limit}
\lim_{n\to+\infty}(1+\Vert u_n\Vert)\Vert\varphi'(u_n)\Vert_{X^*}=0.
\end{equation}
We first show that $\{u_n\}$ is bounded. This evidently happens once the same holds for both $\{ u_n^+\}$ and $\{ u_n^-\}$. \\
\textit{Claim 1: The sequence $\{u^+_n\}$ is bounded.}\\
If the assertion was false then, up to subsequences, $\Vert u_n^+\Vert\to+\infty$. Write $v_n:=\Vert u_n^+\Vert^{-1}u_n^+$. From $\Vert v_n\Vert=1$ it follows, along a subsequence when necessary,
\begin{equation}\label{weak}
v_n\rightharpoonup v\quad\mbox{in $X$},\quad v_n\to v\quad\mbox{in $L^{2N/(N-1)}(\Omega)$ and in $L^2(\partial\Omega)$.}
\end{equation}
Through \eqref{limit} one has $\langle\varphi'(u_n),u_n^+\rangle\to 0$, which, dividing by $\Vert u_n^+\Vert^2$, easily entails
\begin{equation}\label{zero}
{\cal E}(v_n)\leq\ep_n+\int_\Omega\frac{f(x,u_n^+(x))}{\Vert u_n^+\Vert}v_n(x)\, dx\quad\forall\, n\in\mathbb{N},
\end{equation}
where $\ep_n\to 0^+$. Because of $({\rm f}_1)$ the sequence $\{\Vert u_n^+\Vert^{-1} N_f(u_n^+)\}\subseteq L^2(\Omega)$ is bounded. Via the same reasoning made in \cite[pp. 317--318]{MoMoPa} we thus get a function $\theta\in L^{\infty}(\Omega)$ such that $\theta\leq\hat\lambda_1$ and
$$\frac{1}{\Vert u_n^+\Vert} N_f(u_n^+)\rightharpoonup\theta v\quad \mbox{in}\quad L^2(\Omega).$$
Thanks to \eqref{weak}--\eqref{zero} this produces, as $n\to+\infty$,
\begin{equation}\label{us}
{\cal E}(v)-\int_\Omega\theta(x) v(x)^2 dx\leq 0.
\end{equation}
If $\theta\neq\hat\lambda_1$ then, by Lemma \ref{splitting}, $v=0$. Consequently, on account of \eqref{weak}--\eqref{zero} again, $v_n\to 0$ in $X$, which contradicts $\Vert v_n\Vert\equiv1$. Otherwise, from \eqref{us}, \eqref{lone}, $({\rm p}_5)$, and $({\rm p}_6)$ it follows $v=t\hat u_1$ for some $t>0$. So, $u_n^+\to+\infty$ a.e. in $\Omega$. Using $({\rm f}_2)$ and Fatou's Lemma we thus obtain
\begin{equation}\label{absone}
\lim_{n\to+\infty}\int_\Omega[f(x,u_n^+(x))u_n^+(x)-2F(x,u_n^+(x))]\, dx=+\infty.
\end{equation}
On the other hand, \eqref{bounded} forces
$${\cal E}(u_n)-\int_\Omega 2F(x,u_n(x))\, dx\leq 2c_1$$
while \eqref{limit} easily yields
$$-{\cal E}(u_n)+\int_\Omega f(x,u_n(x))u_n(x)\, dx\leq c_2.$$
Therefore, on account of \eqref{rminus},
$$\int_\Omega[f(x,u_n^+)u_n^+-2F(x,u_n^+)]\, dx\leq\int_\Omega[f(x,u_n)u_n-2F(x,u_n)]\, dx\leq 2c_1+c_2\quad\forall\, n\in\mathbb{N},$$ 
which contradicts \eqref{absone}.\\
\textit{Claim 2: The sequence $\{u^-_n\}$  is bounded.}\\
If the assertion was false then, up to subsequences, $\Vert u_n^- \Vert\to+\infty$. Write, like before, $w_n:=\Vert u_n^-\Vert^{-1}u_n^-$. From $\Vert w_n\Vert\equiv 1$ it follows, along a subsequence when necessary,
\begin{equation}\label{weakminus}
w_n\rightharpoonup w\quad\mbox{in $X$},\quad w_n\to w\quad\mbox{in $L^{2N/(N-1)}(\Omega)$ and in $L^2(\partial\Omega)$,}\quad w\geq 0.
\end{equation}
Through \eqref{limit} one has
\begin{equation}\label{aun}
\left|\frac{1}{2}\langle{\cal E}'(u_n),v\rangle-\int_\Omega f(x,u_n)v\, dx\right|\leq\ep_n\Vert v\Vert\quad\forall\, v\in X,
\end{equation}
where $\ep_n\to 0^+$. A simple computation based on  $({\rm f}_1)$ and the  boundedness of $\{u_n^+\}$ immediately leads to
\begin{equation}\label{aunplus}
\left|\frac{1}{2}\langle{\cal E}'(u_n^+),v\rangle-\int_\Omega f(x,u_n^+)v\, dx\right|\leq c_3\Vert v\Vert.
\end{equation}
Since $u_n=u_n^+-u_n^-$, inequalities \eqref{aun}--\eqref{aunplus} produce, after dividing by $\Vert u_n^-\Vert$,
\begin{equation}\label{awn}
\left|\frac{1}{2}\langle{\cal E}'(-w_n),v\rangle-\frac{1}{\Vert u_n^-\Vert}\int_\Omega f(x,-u_n^-)v\, dx\right|\leq\ep'_n\Vert v\Vert,\quad v\in X,
\end{equation}
with $\ep'_n\to 0^+$. Observe next that, by $({\rm f}_1)$ besides \eqref{weakminus},
$$\lim_{n\to+\infty}\frac{1}{\Vert u_n^-\Vert}\int_\Omega f(x,-u_n^-)(w_n-w)\, dx=0.$$
So, \eqref{awn} written for $v:=w_n-w$ provides
$$\lim_{n\to+\infty}\langle{\cal E}'(-w_n),w_n-w\rangle=0,$$
whence
\begin{equation}\label{limwn}
\lim_{n\to+\infty}w_n=w\quad\mbox{in}\quad X
\end{equation}
because, on account of \eqref{weakminus} and \eqref{abeta},
$$\int_\Omega a(x)w_n(x)(w_n(x)-w(x))\, dx\to 0,\quad\int_{\partial\Omega} \beta(x)w_n(x)(w_n(x)-w(x))\, d\sigma\to 0.$$
Thanks to  $({\rm f}_1)$ the sequence $\{\Vert u_n^-\Vert^{-1}N_f(-u_n^-)\} \subseteq L^2(\Omega)$ is bounded. Using the arguments made in \cite[pp. 317--318]{MoMoPa} we thus obtain a function $\alpha\in L^\infty(\Omega)$ such that $\hat\lambda_k \leq\alpha\leq\hat\lambda_{k+1}$ and
$$\frac{1}{\Vert u_n^-\Vert} N_f(-u_n^-)\rightharpoonup -\alpha w\quad\hbox{in}\quad L^2(\Omega).$$
By \eqref{awn}--\eqref{limwn} this implies, as $n\to+\infty$,
$$\frac{1}{2}\langle{\cal E}'(w),v\rangle=\int_\Omega\alpha(x)w(x)v(x)\, dx\quad\forall\, v\in X,$$
i.e., $w$ turns out to be a weak solution of the problem
\begin{equation}\label{absl}
-\Delta u+a(x)u=\alpha(x)u\quad\hbox{in}\quad\Omega,\quad \frac{\partial u}{\partial n}+\beta(x) u=0\quad\hbox{on}\quad \partial\Omega.
\end{equation}
If $\alpha\neq\hat\lambda_k$ and $\alpha\neq\hat\lambda_{k+1}$ then $({\rm p}_9)$ yields                                                                                                                                                                                                                 
$$\hat\lambda_k(\alpha)<\hat\lambda_k(\hat\lambda_k)=1=\hat\lambda_{k+1}(\hat\lambda_{k+1})<\hat\lambda_{k+1}(\alpha).$$
Therefore $w=0$, which contradicts $\Vert w\Vert=1$; cf. \eqref{limwn}. Otherwise, either $\alpha=\hat\lambda_k$ or $\alpha=\hat\lambda_{k+1}$. In both cases, via \eqref{absl} one sees that $w$ has to be nodal, against \eqref{weakminus}.

Summing up, the sequence $\{u_n\}\subseteq X$ is bounded. Along a subsequence when necessary, we may thus assume 
\begin{equation*}
u_n\rightharpoonup u\quad\mbox{in $X$},\quad u_n\to u\quad\mbox{in $L^{2N/(N-1)}(\Omega)$ and in $L^2(\partial\Omega)$,}
\end{equation*}
whence, like before, $u_n\to u$ in  $X$. This completes the proof.
\end{proof}
\subsection{Existence of at least two nontrivial solutions}\label{sect3.1}
Define, provided $x\in\Omega$ and $t,\xi\in\R$,
\begin{equation}\label{truncation}
\hat f_+(x,t):=f(x,t^+)+\hat a t^+,\quad \hat  F_+(x,\xi):=\int_0^\xi \hat f_+(x,t)\, dt.
\end{equation}
%
It is evident that the corresponding truncated functional
\begin{equation*}\label{phiplus}
\hat\varphi_+(u):=\frac{1}{2}\left({\cal E}(u)+\hat a\Vert u\Vert_2^2\right)-\int_\Omega \hat F_+(x,u(x))\, dx,\quad u\in X,
\end{equation*}
belongs to $C^1(X)$ also.
\begin{proposition}\label{coercive}
Let $({\rm f}_1)$--$({\rm f}_2)$ be satisfied. Then $\hat\varphi_+$ is coercive.
\end{proposition}
\begin{proof}
If the conclusion was false, we may construct a sequence $\{u_n\}\subseteq X$ such that $\Vert u_n\Vert\to+\infty$ but
\begin{equation}\label{bone}
\frac{1}{2}\left({\cal E}(u_n)+\hat a\Vert u_n\Vert_2^2\right)-\int_\Omega \hat F_+(x,u_n(x))\, dx\leq c_3\quad\forall\, n\in\mathbb{N}.
\end{equation}
Write $v_n:=\Vert u_n\Vert^{-1}u_n$. From $\Vert v_n\Vert=1$ it follows, along a subsequence when necessary, \eqref{weak}.
Moreover,  by \eqref{truncation}--\eqref{bone},
\begin{equation}\label{btwo}
\frac{1}{2}\left({\cal E}(v_n)+\hat a\Vert v_n^-\Vert_2^2\right)-\frac{1}{\Vert u_n\Vert^2}\int_\Omega F(x,u_n^+(x))\, dx\leq
\frac{c_3}{\Vert u_n\Vert^2},\quad n\in\mathbb{N}.
\end{equation}
Because of $({\rm f}_1)$ the sequence $\{\Vert u_n\Vert^{-2} N_F(u_n^+)\}\subseteq L^1(\Omega)$ is uniformly integrable. Via
the Dunford-Pettis Theorem and the same reasoning made in \cite[pp. 317--318]{MoMoPa} we thus get a function $\theta\in L^{\infty}(\Omega)$ such that $\theta\leq\hat\lambda_1$ and, up to subsequences,
$$\frac{1}{\Vert u_n\Vert^2} N_F(u_n^+)\rightharpoonup\frac{1}{2}\theta (v^+)^2\quad \mbox{in}\quad L^1(\Omega).$$
Using \eqref{btwo}, besides \eqref{weak}, this produces, as $n\to+\infty$,
\begin{equation}\label{bthree}
{\cal E}(v)+\hat a\Vert v^-\Vert_2^2-\int_\Omega\theta(x) v^+(x)^2dx\leq 0,
\end{equation}
whence, in view of Lemma \ref{hatab},
\begin{equation}\label{vplus}
{\cal E}(v^+)\leq\int_\Omega\theta(x)v^+(x)^2dx.
\end{equation}
If $\theta\not\equiv\hat\lambda_1$, then Lemma \ref{splitting} and \eqref{vplus} force $v^+=0$. From \eqref{bthree} and Lemma \ref{hatab} again it follows $v^-=0$, namely $v=0$. Since, like before,
$$\frac{1}{\Vert u_n\Vert^2} N_{\hat F_+}(u_n)\rightharpoonup 0\quad \mbox{in}\quad L^1(\Omega),$$
inequality \eqref{bone} combined with Lemma \ref{hatab} yield $\Vert v_n\Vert\to 0$. However, this is impossible. So, suppose $\theta=\hat\lambda_1$. Gathering \eqref{vplus} and $({\rm p}_4)$--$({\rm p}_6)$ together lead to $v^+=t\hat u_1$ for some $t\geq 0$. The above argument shows that $t>0$. Hence, $v^+>0$ in $\Omega$, which actually means $v=v^+>0$. Recalling the definition of $\{v_n\}$ we thus have $u_n(x)=u_n^+(x)$ as well as $u_n^+(x)\to+\infty$ for every $x\in\Omega$. Consequently, by Lemma \ref{auxone} and Fatou's Lemma,
\begin{equation}\label{abstwo}
\lim_{n\to+\infty}\int_\Omega[\hat\lambda_1u_n(x)^2-2F(x,u_n(x))] dx=+\infty.
\end{equation}
On the other hand, \eqref{truncation}--\eqref{bone}, besides \eqref{lone}, easily give rise to
$$\int_\Omega[\hat\lambda_1u_n(x)^2-2F(x,u_n(x))] dx\leq 2c_3\quad\forall\, n\in\mathbb{N},$$
against \eqref{abstwo}.
\end{proof}
\begin{theorem}\label{thmone}
Under \eqref{abeta}, $({\rm f}_1)$--$({\rm f}_4)$, and the assumption that, for appropriate $a_3>\hat\lambda_1$, 
\begin{equation}\label{useful}
2F(x,\xi)\geq a_3\xi^2\quad\mbox{in}\quad\Omega\times\R^-_0,
\end{equation}
Problem \eqref{prob} possesses at least two nontrivial solutions $u_0\in {\rm int}(C_+)$ and $u_1\in C^1(\overline{\Omega})$.
\end{theorem}
\begin{proof}
A standard argument, which exploits the Sobolev Embedding Theorem and the compactness of the trace operator, ensures that $\hat\varphi_+$ is weakly sequentially lower semi-continuous. Since, due to Proposition \ref{coercive}, it is coercive, we have
\begin{equation}\label{defuzero}
\inf_{u\in X}\hat\varphi_+(u)=\hat\varphi_+(u_0)
\end{equation}
for some $u_0\in X$. Fix $\ep>0$. Assumption  $({\rm f}_4)$ yields $\delta>0$ small such that 
\begin{equation}\label{defdelta}
F(x,\xi)\geq\frac{a_1(x)-\ep}{2}\xi^2\quad\forall\, (x,\xi)\in\Omega\times[-\delta,\delta].
\end{equation}
If $\tau\in (0,1)$ complies with $\tau\hat u_1\leq\delta$, then by \eqref{truncation}, the above inequality, $({\rm p}_6)$, and $({\rm f}_4)$,
\begin{eqnarray*}
\hat\varphi_+(\tau\hat u_1)\leq\frac{\tau^2}{2}\left({\cal E}(\hat u_1)-\int_\Omega(a_1-\ep)\hat u_1^2\, dx\right)
=\frac{\tau^2}{2}\left(\hat\lambda_1\Vert\hat u_1\Vert_2^2-\int_\Omega(a_1-\ep)\hat u_1^2\, dx\right)\\
=\frac{\tau^2}{2}\left(\int_\Omega(\hat\lambda_1-a_1)\hat u^2_1\, dx+\ep\right)<0
\end{eqnarray*}
as soon as $\ep<\int_\Omega(a_1-\hat\lambda_1)\hat u_1^2\, dx$. Hence,
\begin{equation*}
\hat\varphi_+(u_0)<0=\hat\varphi_+(0),
\end{equation*}
which clearly means $u_0\neq 0$. Now, through (\ref{defuzero}) we get $\hat\varphi_+'(u_0)=0$, namely
\begin{equation}\label{fplus}
\int_\Omega\nabla u_0\cdot\nabla v\, dx+\int_\Omega(a+\hat a)u_0v\, dx+\int_{\partial\Omega}\beta u_0v\, d\sigma=
\int_\Omega[f(x,u_0^+)+\hat au_0^+]v\, dx,\quad v\in X.
\end{equation}
Using Lemma \ref{hatab} and \eqref{fplus} written for $v:=-u_0^-$ produces
$$\hat b\Vert u_0^-\Vert^2\leq {\cal E}(u_0^-)+\hat a\Vert u_0^-\Vert_2^2=-\int_{\partial\Omega}\beta(x)u_0^-(x)^2d\sigma
\leq 0,$$
whence $u_0\geq 0$. Therefore, $u_0$ is a nontrivial nonnegative solution to \eqref{prob}, because \eqref{fplus} becomes
$$\int_\Omega\nabla u_0\cdot\nabla v\, dx+\int_\Omega au_0v\, dx+\int_{\partial\Omega}\beta u_0v\, d\sigma=
\int_\Omega f(x,u_0)v\, dx\quad\forall\, v\in X.$$
By Lemma \ref{auxtwo} one has $u_0\in{\rm int}(C_+)$ while \eqref{truncation} forces $\varphi|_{C_+}=\hat\varphi_+|_{C_+}$. Thus, \cite[Proposition 3]{PaRa} ensures that $u_0$ turns out to be a local minimizer for $\varphi$. We may evidently assume $u_0$
isolated in $K(\varphi)$, otherwise infinitely many solutions there would exist. The same reasoning made in the proof of \cite[Proposition 29]{APS2} provides here $\rho>0$ fulfilling
\begin{equation}\label{crho}
\varphi(u_0)<c_\rho:=\inf_{u\in\partial B_\rho(u_0)}\varphi(u).
\end{equation}
From \eqref{rminus} it easily follows
$$\lim_{\tau\to-\infty}\varphi(\tau\hat u_1)=-\infty.$$
Thanks to Proposition \ref{cerami}, Condition (C) holds true for $\varphi$. Consequently, the Mountain Pass Theorem gives $u_1\in X\setminus\{u_0\}$ such that $\varphi'(u_1)=0$ and
\begin{equation}\label{defw}
c_\rho\leq\varphi(u_1)= \inf_{\gamma\in\Gamma}\max_{t\in[0,1]}\varphi(\gamma(t)),
\end{equation}
where
$$\Gamma:=\{\gamma\in C^0([0,1],X):\;\gamma(0)=-\tau\hat u_1,\;\gamma(1)=u_0\}$$
with sufficiently large $\tau>0$.  Obviously, $u_1$ solves (\ref{prob}). Through the regularity arguments exploited in the proof of Lemma \ref{auxtwo} we achieve $u_1\in C^1(\overline{\Omega})$. Thus, the only thing to check is that $u_1\neq 0$. This will be a consequence of the inequality
\begin{equation}\label{wminore}
\varphi(u_1)<0\, ,
\end{equation}
which, due to (\ref{defw}), derives from the claim below.  
\begin{equation}\label{path}
\mbox{There exists a path $\tilde\gamma \in\Gamma$ such that $\varphi(\tilde\gamma(t))<0$ for all $t\in [0,1]$.}
\end{equation}
Pick $\ep>0$ and choose $\delta>0$ as in \eqref{defdelta}. Combining (${\rm p}_8$) with \cite[Lemma 2.1]{MaPaRobin} entails
\begin{equation}\label{smalleta}
\max_{t\in [-1,1]}{\cal E}(\gamma_\ep(t))<\hat\lambda_2+\ep
\end{equation}
for appropriate $\gamma_\ep\in C^0([-1,1],C^1(\overline{\Omega}))\cap\Gamma_1$. Since $\gamma_\ep([-1,1])$ is compact in $C^1(\overline{\Omega})$ while $\tau\hat u_1,u_0 \in{\rm int}(C_+)$ we can find $\eta>0$ so small that
$$-\tau\hat u_1(x)\leq\eta\gamma_\ep(t)(x)\leq u_0(x),\quad|\eta\gamma_\ep (t)(x)|\leq\delta$$
whenever $x\in\Omega$, $t\in[-1,1]$. On account of \eqref{defdelta}, \eqref{smalleta}, and the equality $\Vert\gamma_\ep(t)\Vert_2=1$ one has
\begin{eqnarray*}
\varphi(\eta\gamma_\ep(t))=\frac{\eta^2}{2}{\cal E}(\gamma_\ep(t))-\int_\Omega F(x,\eta\gamma_\ep(t)(x))\, dx\\
\phantom{}\\
<\frac{\eta^2}{2}\left(\hat\lambda_2+\ep+\int_{\Omega}(\ep-a_1)|\gamma_\ep(t)|^2dx\right)
<\frac{\eta^2}{2}\left(\int_\Omega(\hat\lambda_2-a_1)|\gamma_\ep(t)|^2dx+2\ep\right)<0
\end{eqnarray*}
provided $\ep<2^{-1}\int_\Omega(a_1-\hat\lambda_2)|\gamma_\ep(t)|^2dx$. Consequently,
\begin{equation}\label{middle}
\varphi|_{\eta\gamma_\ep([-1,1])}<0.
\end{equation}
Next, write $a:=\hat\varphi_+(u_0)$. It is evident that $a<0$. Further, we may suppose
$$K(\hat\varphi_+)=\{0,u_0\},$$
otherwise the conclusion would be straightforward. Hence, no critical value of $\hat\varphi_+$ lies
in $(a,0)$ while
$$K_a(\hat\varphi_+) =\{ u_0\}.$$
Due to the Second Deformation Lemma \cite[Theorem 5.1.33]{GaPaNA}, there exists a continuous function $h:[0,1]\times (\hat\varphi_+^0\setminus\{0\})\to\hat\varphi_+^0$ satisfying
$$h(0,u)=u\, ,\quad h(1,u)=u_0\, ,\quad\mbox{and}\quad\hat\varphi_+(h(t,u))\leq\hat\varphi_+(u)$$
for all $(t,u)\in [0,1]\times(\hat\varphi_+^0\setminus\{0\})$. Let $\gamma_+(t):=h(t,\eta\hat u_1)^+$, $t\in [0,1]$. Then $\gamma_+(0)=\eta\hat u_1$, $\gamma_+(1)=u_0$, as well as
\begin{equation}\label{gammapiu}
\varphi(\gamma_+(t))=\hat\varphi_+(\gamma_+(t))\leq\hat\varphi_+(h(t,\eta\hat u_1)) \leq\hat\varphi_+(\eta\hat u_1)=\varphi(\eta\gamma_\ep(1))<0;
\end{equation}
cf. \eqref{middle} besides \eqref{truncation} and Lemma \ref{hatab}.  Finally, define
$$\gamma_-(t):=-[t\eta+(1-t)\tau]\hat u_1\quad\forall\, t\in [0,1].$$
Since through \eqref{useful} we obtain
\begin{equation}\label{gammameno}
\varphi(\gamma_-(t))\leq\frac{1}{2}\left({\cal E}(\gamma_-(t))-a_3\Vert\gamma_-(t)\Vert_2^2\right)
=\frac{1}{2}[t\eta+(1-t)\tau]^2(\hat\lambda_1-a_3)<0,
\end{equation}
concatenating $\gamma_-$, $\eta\gamma_\ep$, and $\gamma_+$ produces a path $\tilde\gamma \in\Gamma$ which, in view of (\ref{middle})--(\ref{gammameno}), fulfils (\ref{path}).
\end{proof}
A variant of  Theorem \ref{thmone} that does not change the overall problem's geometry is the one below, where
\begin{itemize}
\item[$({\rm f}_5)$] \textit{There exist $m\geq\max\{n_0,2\}$ as well as  $a_1,a_2\in L^\infty(\Omega)\setminus\{\hat\lambda_m,\hat\lambda_{m+1}\}$ such that $\hat\lambda_m\leq a_1\leq a_2\leq\hat\lambda_{m+1}$ and
$$a_1(x)\leq\liminf_{t\to 0}\frac{f(x,t)}{t}\leq\limsup_{t\to 0}\frac{f(x,t)}{t}\leq a_2(x)$$
uniformly in $x\in\Omega$.}
\item[$({\rm f}_6)$] \textit{$f(x,\cdot)$ is differentiable at zero and $f'_t(x,0)=\displaystyle{\lim_{t\to 0}\frac{f(x,t)}{t}}$ uniformly with respect to $x\in\Omega$. Moreover, for appropriate $m\geq \max\{n_0,2\}$ and $a_2\in L^\infty(\Omega)\setminus\{\hat \lambda_{m+1}\}$ one has
$$\hat\lambda_m\leq f'_t(\cdot,0)\leq a_2\leq\hat\lambda_{m+1}.$$ }
\end{itemize}

\begin{lemma}\label{cqphi}
If $({\rm f}_1)$, $({\rm f}_2)$, $({\rm f}_3)$, and either $({\rm f}_5)$ or  $({\rm f}_6)$ hold true then $C_q(\varphi,0)=\delta_{q,d_m}\mathbb{Z}$ for all $q\in\mathbb{N}_0$, where $d_m:={\rm dim}(\bar H_m)$.
\end{lemma}
\begin{proof}
\textit{1) Under Condition $({\rm f}_5)$.}\\
Pick any $\theta\in L^\infty(\Omega)$ satisfying $a_1\leq\theta\leq a_2$ and set
$$\psi(u):=\frac{1}{2}\left({\cal E}(u)-\int_\Omega\theta(x)u(x)^2dx\right),\quad u\in X.$$
Thanks to $({\rm f}_5)$, Lemma \ref{splitting} can be applied. So, $u=0$ is a non-degenerate critical point of $\psi$ with Morse index $d_m$, which forces
\begin{equation}\label{Cqpsi}
C_q(\psi,0)=\delta_{q,d_m}\mathbb{Z}\quad\forall\, q\in\mathbb{N}_0;
\end{equation}
see \eqref{kd}. Now, let $h:[0,1]\times X\to\R$ given by
$$h(t,u):=(1-t)\varphi(u)+t\psi(u),\quad (t,u)\in[0,1]\times X.$$
We shall prove that there exists $r>0$ such that $0\not\in h([0,1]\times\partial B_r)$. In fact, if not, one might construct two sequences $\{t_n\}\subseteq [0,1]$ and $\{u_n\}\subseteq X\setminus\{0\}$ with the properties
$$t_n\to t\in [0,1],\quad u_n\to 0\mbox{ in } X,\quad h'_u(t_n,u_n)=0\;\;\forall\, n\in\mathbb{N}.$$
Consequently, letting $v_n:=\Vert u_n\Vert^{-1}u_n$ we have
\begin{equation}\label{avnw}
\int_\Omega\left(\nabla v_n\cdot\nabla w+av_nw\right)dx+\int_{\partial\Omega}\beta v_nwd\sigma=\int_\Omega \left[(1-t_n)\frac{f(x,u_n)}{\Vert u_n\Vert}+t_n\theta v_n\right]wdx
\end{equation}
whatever $w\in X$ as well as \eqref{weak}. Because of $({\rm f}_1)$ the sequence $\{\Vert u_n\Vert^{-1} N_f(u_n)\}\subseteq L^2(\Omega)$ is bounded. The same reasoning made in \cite[pp. 317--318]{MoMoPa} produces a function $\hat\theta\in L^{\infty}(\Omega)$ such that $a_1\leq\hat\theta\leq a_2$ and
\begin{equation}\label{hattheta}
\frac{1}{\Vert u_n\Vert} N_f(u_n)\rightharpoonup\hat\theta v\quad \mbox{in}\quad L^2(\Omega).
\end{equation}
Thanks to \eqref{avnw} besides \eqref{weak} we get, as $n\to+\infty$,
\begin{equation*}
\int_\Omega\left(\nabla v\cdot\nabla w+avw\right)dx+\int_{\partial\Omega}\beta vw\, d\sigma=\int_\Omega[(1-t)\hat\theta +t\theta]vw\, dx,\quad w\in X,
\end{equation*}
namely $v$ is a weak solution of the problem
$$-\Delta u+a(x)u=\theta_t(x) u\quad\mbox{in}\quad\Omega,\quad\frac{\partial u}{\partial n}+\beta(x) u=0\quad\mbox{on}\quad\partial\Omega,$$
where $\theta_t(x):=(1-t)\hat\theta(x)+t\theta(x)$. From $a_1\leq\theta_t\leq a_2$, assumption $({\rm f}_5)$, and  $({\rm p}_9)$ it follows
$$\hat\lambda_m(\theta_t)<\hat\lambda_m(\hat\lambda_m)=1=\hat\lambda_{m+1}(\hat\lambda_{m+1})<\hat\lambda_{m+1}(\theta_t),$$
whence $v=0$. Through \eqref{weak}, \eqref{hattheta}, and \eqref{avnw} written for $w:=v_n-v$  we easily infer that $v_n\to 0$ in $X$, but this is impossible, because $\Vert v_n\Vert\equiv 1$. Finally, combining the homotopy invariance property of critical  groups with \eqref{Cqpsi} completes the proof.

\textit{2) Under Condition $({\rm f}_6)$.}\\
Define, like before,
$$\psi(u):=\frac{1}{2}\left({\cal E}(u)-\int_\Omega f'_t(x,0)u(x)^2dx\right),\quad u\in X.$$
Thanks to $({\rm f}_6)$ and $({\rm p}_4)$ one has $\psi(u)\leq 0$ for all $u\in\bar H_m$. If $\hat c>0$ is furnished by Lemma \ref{splitting} for $n:=m+1$ then
\begin{equation*}
\psi(u)\geq\frac{1}{2}\left({\cal E}(u)-\int_\Omega a_2(x)u(x)^2dx\right)\geq\frac{\hat c}{2}\Vert u\Vert^2>0\quad\forall\, u\in\hat H_{m+1}\setminus\{0\},
\end{equation*}
since $f'_t(x,0)t^2\leq a_2(x)t^2$ in $\Omega\times\R$. Now, Proposition 2.3 of \cite{Su} ensures that \eqref{Cqpsi} holds.  Due to  $({\rm f}_6)$ again, given any $\ep>0$ we can find $\delta>0$ fulfilling
\begin{equation*}
|f(x,t)-f'_t(x,0)t|\leq\ep|t|,\quad (x,t)\in\Omega\times[-\delta,\delta].
\end{equation*}
This entails
$$|\varphi(u)-\psi(u)|\leq\int_\Omega\left(\int_0^{|u(x)|}|f(x,t)-f'_t(x,0)t|d|t|\right)dx\leq\frac{\ep}{2}\Vert u\Vert_2^2$$
as well as
$$|\langle\varphi'(u)-\psi'(u),v\rangle|\leq\int_\Omega|f(x,u)-f'_t(x,u)u||v|dx\leq\ep\Vert u\Vert_2\Vert v\Vert_2\quad\forall\, v\in X$$
provided $|u(x)|\leq\delta$. Consequently, to every $\ep>0$ there corresponds  $\rho>0$ such that
$$\Vert\varphi-\psi\Vert_{C^1(D_\rho)}\leq\ep,$$
where $D_\rho:=\{u\in C^1(\overline{\Omega}):\Vert u\Vert_{C^1(\overline{\Omega})}\leq\rho\}$. Corollary 5.1.25 of \cite{ChMNA} thus yields
\begin{equation*}
C_q(\varphi|_{C^1(\overline{\Omega})},0)=C_q(\psi|_{C^1(\overline{\Omega})},0),\quad q\in\mathbb{N}_0,
\end{equation*}
which actually means $C_q(\varphi,0)=C_q(\psi,0)$, because $C^1(\overline{\Omega})$ is dense in $X$; see, e.g., \cite{Palais}. Now the conclusion directly follows from \eqref{Cqpsi}.
\end{proof}
\begin{theorem}\label{thmtwo}
Let \eqref{abeta}, $({\rm f}_1)$--$({\rm f}_3)$, and either $({\rm f}_5)$ or $({\rm f}_6)$ be satisfied. Then the same conclusion of Theorem \ref{thmone} holds.
\end{theorem}
\begin{proof}
An argument analogous to that employed in showing Theorem \ref{thmone} provides here two solutions, $u_0\in{\rm int}(C_+)$ and $u_1\in C^1(\overline{\Omega})$. So, it remains to see whether $u_1\neq 0$. By \cite[Proposition 6.100]{MoMoPa} we have $C_1(\varphi,u_1)\neq 0$ while Lemma \ref{cqphi} entails $C_1(\varphi,0)=\delta_{1,d_m}\mathbb{Z}$. Since $d_m\geq 2$, the function $u_1$ cannot be trivial.
\end{proof}
\begin{remark}
 Although $({\rm f}_5)$ and $({\rm f}_6)$ look less general than $({\rm f}_4)$, inequality \eqref{useful} is not taken on. 
\end{remark}
The next variant of Theorem \ref{thmone} exhibits a different geometry at zero. Indeed, instead of $({\rm f}_4)$, $({\rm f}_5)$, or $({\rm f}_6)$, we shall suppose
\begin{itemize}
\item[$({\rm f}_7)$] \textit{There exist $a_4>0$, $q\in(0,2)$, and $\delta>0$ such that
$$a_4|t|^q\leq f(x,t)t\leq qF(x,t)\quad\forall\, (x,t)\in\Omega\times[-\delta,\delta].$$
}
\end{itemize}
Condition $({\rm f}_7)$ allows to get further information on the critical groups of $\varphi$ at zero. This has previously been pointed out in \cite{Mo} concerning a different problem; cf. also \cite{MaMoPa}.
\begin{lemma}\label{Cqzero}
Under \eqref{abeta}, $({\rm f}_1)$, $({\rm f}_7)$, and the assumption that zero is an isolated critical point of $\varphi$, one has $C_q(\varphi, 0)=0$ for all $q\in\mathbb{N}_0$.
\end{lemma}
\begin{proof}
Let $\psi:X\to\R$ be defined by
$$\psi(u):=\frac{1}{2}\left({\cal E}(u)+\hat a\Vert u\Vert^2_2\right)-\int_\Omega F(x,u(x)) dx,\quad u\in X.$$
Obviously, zero turns out to be an isolated critical point of $\psi$, because
$$\Vert\psi'(u)-\varphi'(u)\Vert_{X^*}\leq\hat a\Vert u\Vert_2.$$
Reasoning as in the proof of Lemma \ref{cqphi} we get $C_q(\varphi,0)=C_q(\psi,0)$. Thus, the conclusion is achieved once
\begin{equation*}
C_q(\psi,0)=0,\quad q\in\mathbb{N}_0.
\end{equation*}
Thanks to  $({\rm f}_1)$ and $({\rm f}_7)$, given any $r\in(2,2^*)$, there exists $c_4>0$ fulfilling
\begin{equation}\label{double}
F(x,t)\geq\frac{a_4}{q}|t|^q-c_4|t|^r\;\mbox{ and }\; qF(x,t)-f(x,t)t\geq-c_4|t|^r\;\mbox{ in }\;\Omega\times\R.
\end{equation}
If $u\in (X\setminus\{0\})\cap\psi_0$ then
\begin{eqnarray*}
\frac{d}{d\tau}\left.\psi(\tau u)\right|_{\tau=1}=\langle \psi'(u), u \rangle\geq\langle \psi'(u), u \rangle-q\psi(u)\\
=\left(1-\frac{q}{2}\right)\left({\cal E}(u)+\hat a\Vert u\Vert_2^2\right)+\int_\Omega[qF(x,u)- f(x, u)u]\, dx\\
\geq\left(1-\frac{q}{2}\right)\hat b\Vert u\Vert^2-c_4\Vert u\Vert^r
\end{eqnarray*}
by \eqref{double} besides  Lemma \ref{hatab}. Consequently
$$\frac{d}{d\tau}\left.\psi(\tau u)\right|_{\tau=1}>0$$
whenever $\Vert u\Vert$ is sufficiently small, say $u\in(\bar B_{2\rho}\setminus\{0\})\cap\psi_0$ for some $\rho>0$. Hence, in particular, $\tau_0>0$ and $\tau_0u\in(\bar B_{2\rho}\setminus\{0\})\cap\psi_0$ imply
$$\frac{d}{d\tau}\left.\psi(\tau u)\right|_{\tau=\tau_0}=\frac{1}{\tau_0}\frac{d}{d\tau}\left.\psi(\tau \tau_0u)\right|_{\tau=1} >0.$$
This means that the $C^1$-function $\tau\mapsto\psi(\tau u)$, $\tau\in (0,+\infty)$, turns out to be increasing at the point $\tau$ provided $\tau u\in(\bar B_{2\rho}\setminus\{0\})\cap\psi_0$. So, it vanishes at most once in the open interval $(0,2\Vert u\Vert^{-1}\rho)$. On the other hand, \eqref{double} yields
$$\psi(\tau u)\leq\frac{\tau^2}{2}\left({\cal E}(u)+\hat a\Vert u\Vert_2^2\right)-a_4\frac{\tau^q}{q}\Vert u\Vert_q^q+c_4\tau^r\Vert u\Vert_r^r,$$
whence $\psi(\tau u)<0$ for all $\tau>0$ small enough, since $q<2<r$. Summing up, given any $u\in\bar B_{2\rho}\setminus\{0\}$, either $\psi(\tau u)<0$ as soon as $\tau u\in\bar B_{2\rho}$ or
\begin{equation}\label{btau}
\mbox{$\exists$ a unique $\bar\tau(u)>0$ such that $\bar\tau(u)u\in\bar B_{2\rho}\setminus\{0\}$,
$\psi(\bar\tau(u)u)=0$.}
\end{equation}
Moreover, if $u\in(\bar B_{2\rho}\setminus\{0\})\cap\psi_0$ then $0<\bar\tau(u)\leq 1$ and
\begin{equation*}
\psi(\tau u)<0\;\;\forall\,\tau\in(0,\bar\tau(u)),\quad\psi(\tau u)>0\;\;\forall\,\tau>\bar\tau(u)\;\;\mbox{with}\;\;\tau u\in\bar B_{2\rho}.
\end{equation*}
Let $\tau:\bar B_\rho\setminus\{0\}\to (0,+\infty)$ be defined by
$$\tau(u):=\left\{
\begin{array}{ll}
1 & \mbox{when $u\in(\bar B_\rho\setminus\{0\})\cap\psi^0$,}\\
\bar\tau(u) & \mbox{when $u\in(\bar B_\rho\setminus\{0\})\cap\psi_0$.}
\end{array}
\right.$$
We claim that the function $\tau(u)$ is continuous. This immediately follows once one knows that  $\bar\tau(u)$ turns out to be continuous on $(\bar B_\rho\setminus\{0\})\cap\psi_0$, because, by uniqueness, $u\in\bar B_\rho\setminus\{0\}$ and $\psi(u)=0$ evidently force $\bar\tau(u)=1$; cf. (\ref{btau}). Pick $\hat u\in(\bar B_\rho\setminus\{0\}) \cap\psi_0$. The function $\phi(t,u):=\psi(tu)$ belongs to $C^1(\R\times X)$ and, on account of (\ref{btau}), we have
$$\phi(\bar\tau(\hat u),\hat u)=0,\quad\frac{\partial\phi}{\partial u}(\bar\tau(\hat u),\hat u)
=\bar\tau(\hat u)\psi'(\bar\tau(\hat u)\hat u).$$
Since zero turns out to be an isolated critical point for $\psi$, there is no loss of generality in assuming $K_\varphi\cap \bar B_\rho=\{0\}$. So, the Implicit Function Theorem furnishes $\ep>0$, $\sigma\in C^1(B_\ep (\hat u))$ such that
$$\phi(\sigma(u),u)=0\;\;\forall\, u\in B_\ep(\hat u),\quad\sigma(\hat u)=\bar\tau(\hat u).$$
Through $0<\bar\tau(\hat u)\leq 1$ we thus get $0<\sigma(u)<2$ for all $u\in U$, where $U\subseteq B_\ep(\hat u)$ denotes a convenient  neighborhood of $\hat u$. Consequently,
$$\sigma(u)u\in\bar B_{2\rho}\setminus\{0\}\;\;\mbox{and}\;\;\psi(\sigma(u)u)=0\;\;\mbox{provided}\;\;u\in(\bar B_\rho \setminus\{0\})\cap\psi_0\cap U. $$
By (\ref{btau}) this results in $\sigma(u)=\bar\tau(u)$, from which the continuity of $\bar\tau(u)$ at $\hat u$ follows.  As $\hat u$ was arbitrary, the function $\bar\tau(u)$ turns out to be continuous on $(\bar B_\rho\setminus\{0\})\cap\psi_0$.

Next, observe that $\tau u\in\bar B_{\rho}\cap\psi^0$ for all $\tau\in[0,1]$, $u\in\bar B_{\rho}\cap \psi^0$. Hence, if
$$ h(t, u):=(1-t)u,\quad (t,u)\in [0,1]\times (\bar B_\rho\cap \psi^0)\, $$
then $h([0,1]\times (\bar B_\rho\cap \psi^0))\subseteq\bar B_\rho\cap \psi^0$, namely $\bar B_\rho\cap \psi^0$ is contractible in itself. Moreover, the function
$$g(u):=\tau(u)u\quad\forall\, u\in\bar B_\rho\setminus\{0\}$$
is continuous and one has $g(\bar B_\rho\setminus\{0\})\subseteq (\bar B_\rho\cap\psi^0)\setminus\{0\}$. Since
$$g|_{ (\bar B_\rho\cap \psi^0)\setminus\{0\}}={\rm id}|_{ (\bar B_\rho\cap \psi^0)\setminus\{0\}}\, ,$$
the set $(\bar B_\rho\cap \psi^0)\setminus\{0\}$ turns out to be a retract of $\bar B_\rho\setminus\{0\}$. Being $\bar B_\rho \setminus\{0\}$ contractible in itself because $X$ is infinite dimensional, we get (see, e.g., \cite[p. 389]{GrDu})
$$C_q(\psi, 0):= H_q(\bar B_\rho\cap \psi^0,(\bar B_\rho\cap \psi^0)\setminus\{0\})=0\, ,\quad q\in\mathbb{N}_0\, ,$$
as desired.
\end{proof}
\begin{remark}
This proof is patterned after that of \cite[Theorem 3.1]{MaMoPa}.
\end{remark}
\begin{theorem}\label{thmthree}
Let \eqref{abeta}, $({\rm f}_1)$--$({\rm f}_3)$, and  $({\rm f}_7)$ be satisfied. Then the same conclusion of Theorem \ref{thmone} holds.
\end{theorem}
\begin{proof}
Reasoning exactly as in the proof of the above-mentioned result yields \eqref{defuzero} for some $u_0\in X$. Furthermore, with $({\rm f}_4)$ replaced by $({\rm f}_7)$, one achieves both $u_0\neq 0$ and $u_0\in{\rm int}(C_+)$; cf. Lemma \ref{auxtwo}. So, $u_0$ turns out to be a local minimizer for $\varphi$, which entails $u_0\in K(\varphi)$. Proposition \ref{cerami} guarantees that $\varphi$ fulfils Condition (C). Thus, the arguments exploited in the proof of Theorem \ref{thmone} provide a second solution $u_1\in C^1(\overline{\Omega})$. Thanks to \cite[Proposition 6.100]{MoMoPa} we have $C_1(\varphi,u_1)\neq 0$. Since $C_1(\varphi,0)=0$ by Lemma \ref{Cqzero}, the function $u_1$ cannot be zero.
\end{proof}
\subsection{Existence of at least three nontrivial solutions}
From now on, we shall suppose that $f(\cdot,0)=0$, $f(x,\cdot)\in C^1(\R)$ for every $x\in\Omega$, and
\begin{itemize}
\item[$({\rm f}'_1)$]  \textit{There exists $a_0\in L^\infty(\Omega)$ satisfying $|f'_t(x,t)|\leq a_0(x)$ in $\Omega\times\R$.}
\end{itemize}
\begin{lemma}\label{cqinfty}
Under \eqref{abeta}, $({\rm f}'_1)$, $({\rm f}_2)$, and $({\rm f}_3)$, one has $C_q(\varphi,\infty)=0$ for all $q\in\mathbb{N}_0$.
\end{lemma}
\begin{proof}
Pick $\mu\in(\hat\lambda_k,\hat\lambda_{k+1})$. Define, provided $(t,u)\in [0,1]\times X$,
$$h(t,u):=\frac{1}{2}{\cal E}(u)-t\int_\Omega F(x,u)dx+\frac{1-t}{2}\left(\hat a\Vert u^+\Vert_2^2-\mu\Vert u^-\Vert_2^2 \right).$$
Clearly, $h$ maps bounded sets into bounded sets. On account of Proposition \ref{cerami}, both $h(0,\cdot)$ and $h(1,\cdot)$ satisfy Condition (C). Due to $({\rm f}'_1)$, the functionals $t\mapsto h'_t(t,u)$ and $u\mapsto h'_u(t,u)$ are locally Lipschitz continuous. Let us next verify that 
\begin{equation}\label{furtherclaim}
\exists\, a\in\R,\, \delta>0\;\mbox{ fulfilling }\; h(t,u)\leq a\;\implies\; (1+\Vert u\Vert)\Vert h'_u(t,u)\Vert_{X^*}
\geq\delta\Vert u\Vert^2.
\end{equation}
If the assertion were false, then we might find two sequences $\{t_n\}\subseteq [0,1]$, $\{u_n\}\subseteq X$ with the properties below:
\begin{eqnarray}\label{tnun}
t_n\to t,\quad\Vert u_n\Vert\to+\infty,\quad h(t_n,u_n)\to-\infty,\nonumber\\
\phantom{}\\
 (1+\Vert u_n\Vert)\Vert h'_u(t_n,u_n)\Vert_{X^*}<\frac{1}{n}\Vert u_n\Vert^2\quad\forall\, n\in\mathbb{N}.\nonumber
\end{eqnarray}
Put $v_n:=\Vert u_n\Vert^{-1}u_n$. Reasoning as in the proof of Lemma \ref{cqphi} produces $v\in X$ such that $v_n\to v$ and
\begin{equation}\label{long}
\int_\Omega\left[\nabla v\cdot\nabla w+avw+(1-t)\hat av^+w\right]dx+\int_{\partial\Omega}\beta vwd\sigma=\int_\Omega\left[ t\theta v^+-\eta_t v^-\right]wdx
\end{equation} 
for all $w\in X$, where $\eta_t(x):=t\eta(x)+(1-t)\mu$ while $\theta,\eta\in L^\infty(\Omega)$ comply with $\theta\leq\hat \lambda_1$, $\hat\lambda_k\leq\eta\leq\hat\lambda_{k+1}$. Hence, $v$ is a weak solution to the problem
$$-\Delta u+a(x)u+(1-t)\hat au^+=t\theta(x)u^+-\eta_t(x)u^-\quad\mbox{in $\Omega$,}\quad\frac{\partial u}{\partial n}+\beta(x)u=0\quad\mbox{on $\partial\Omega$.}$$
If $t=1$ then \eqref{long} written for $w:=v^+$ entails
$${\cal E}(v^+)=\int_\Omega\theta(x)v^+(x)^2dx.$$
Two situations may now occur:\\
1) $\theta\not\equiv\hat\lambda_1$. Lemma \ref{splitting} immediately forces $v^+=0$. Consequently,
$$-\Delta v+a(x)v=\eta(x)v\quad\mbox{in $\Omega$,}\quad\frac{\partial v}{\partial n}+\beta(x)v=0\quad\mbox{on $\partial\Omega$.}$$
Since $v\neq 0$, because $\Vert v\Vert=1$, and $\hat\lambda_k\leq\eta\leq\hat\lambda_{k+1}$, through $({\rm p}_9)$ we see that $v$ must change sign, which is absurd.\\
2) $\theta=\hat\lambda_1$. Likewise the proof of Proposition \ref{coercive}, \eqref{tnun} give rise to a contradiction.\\
Therefore, $t<1$. Letting $w:=v^+$ in \eqref{long} yields
$$(1-t){\cal E}(v^+)+(1-t)\hat a\Vert v^+\Vert_2^2=t\left[\int_\Omega\theta(v^+)^2dx-{\cal E}(v^+)\right].$$
From Lemmas \ref{hatab}--\ref{splitting} it thus follows $(1-t)\hat b\Vert v^+\Vert^2\leq 0$, whence $v=-v^-$. Now, \eqref{long}
becomes
$$-\Delta v+a(x)v=\eta_t(x)v\quad\mbox{in $\Omega$,}\quad\frac{\partial v}{\partial n}+\beta(x)v=0\quad\mbox{on $\partial\Omega$,}$$
and, as before, $v$ has to be nodal, since $\hat\lambda_k\leq\eta_t\leq\hat\lambda_{k+1}$ by the choice of $\mu$. However, this is impossible. Thus, \eqref{furtherclaim} holds true. Via Proposition \ref{abstract} we obtain
\begin{equation}\label{cqhinfty}
C_q(\varphi,\infty)=C_q(h(1,\cdot),\infty)=C_q(h(0,\cdot),\infty)\quad\forall\, q\in\mathbb{N}_0.
\end{equation}
Observe next that
\begin{equation}\label{cqhzero}
C_q(h(0,\cdot),\infty)=C_q(h(0,\cdot),0).
\end{equation}
In fact, if $u\in K(h(0,\cdot))$ then
\begin{equation}\label{short}
\int_\Omega\left(\nabla u\cdot\nabla v+auv+\hat au^+v\right)dx+\int_{\partial\Omega}\beta uvd\sigma=-\mu\int_\Omega
u^-vdx,\quad v\in X.
\end{equation} 
Choosing $v:=u^+$ furnishes ${\cal E}(u^+)+\hat a\Vert u^+\Vert_2^2=0$, namely $u^+=0$; cf. Lemma \ref{hatab}. So, \eqref{short} actually means
$$-\Delta u+a(x)u=\mu u\quad\mbox{in $\Omega$,}\quad\frac{\partial u}{\partial n}+\beta(x)u=0\quad\mbox{on $\partial\Omega$,}$$
and, a fortiori, $u=0$, because $\hat\lambda_k<\mu<\hat\lambda_{k+1}$. This shows that $ K(h(0,\cdot))=\{0\}$, from which \eqref{cqhzero} follows at once.\\
Let us finally compute $C_q(h(0,\cdot),0)$. Consider the homotopy
$$\hat h(t,u):=h(0,u)+t\int_\Omega u(x)dx\quad\forall\, (t,u)\in [0,1]\times X.$$
We claim that
\begin{equation}\label{claim2}
\hat h'_u(t,u)\neq 0,\quad (t,u)\in[0,1]\times(X\setminus\{0\}).
\end{equation}
By contradiction, suppose there exists $(t,u)\in (0,1]\times(X\setminus\{0\})$ fulfilling $h'_u(t,u)=0$. The same arguments exploited above produce here $u\leq 0$ and
$$-\Delta u+a(x)u=\mu u-t\quad\mbox{in $\Omega$,}\quad\frac{\partial u}{\partial n}+\beta(x)u=0\quad\mbox{on $\partial\Omega$.}$$
Hence, likewise the proof of Lemma \ref{auxtwo}, $u\in-{\rm int}(C_+)$. Define, for every $v\in{\rm int}(C_+)$,
$$R(v,-u):=|\nabla v|^2-\nabla(-u)\cdot\nabla\left(\frac{v^2}{-u}\right).$$
Using Picone's identity \cite[Proposition 9.61]{MoMoPa} yields
\begin{eqnarray*}
0\leq\int_\Omega R(v,-u)(x)dx=\Vert\nabla v\Vert_2^2-\int_\Omega(-\Delta u)\frac{v^2}{u}dx-\int_{\partial\Omega}\beta u\frac{v^2}{-u}d\sigma\\
=\Vert\nabla v\Vert_2^2+\int_\Omega av^2dx+\int_{\partial\Omega}\beta v^2d\sigma-\mu\int_\Omega v^2dx +t\int_\Omega\frac{v^2}{u}dx\\
<\Vert\nabla v\Vert_2^2+\int_\Omega av^2dx+\int_{\partial\Omega}\beta v^2d\sigma-\mu\int_\Omega v^2dx.
\end{eqnarray*}
On account of $({\rm p}_6)$ this entails, for $v:=\hat u_1$, 
$$0<\hat\lambda_1-\mu<0,$$
which is clearly absurd. Thanks to \eqref{claim2} and Theorem 5.1.2 in \cite{ChMNA} we have, for $\rho>0$ small enough,
\begin{equation}\label{hathone}
H_q(\hat h(1,\cdot)^0\cap B_\rho,\hat h(1,\cdot)^0\cap B_\rho\setminus\{0\})=0
\end{equation}
while the homotopy invariance of singular homology forces
\begin{equation}\label{hathzero}
H_q(\hat h(0,\cdot)^0\cap B_\rho,\hat h(0,\cdot)^0\cap B_\rho\setminus\{0\})=
H_q(\hat h(1,\cdot)^0\cap B_\rho,\hat h(1,\cdot)^0\cap B_\rho\setminus\{0\}).
\end{equation}
Since $\hat h(0,\cdot)=h(0,\cdot)$, \eqref{hathone}--\eqref{hathzero} provide
\begin{equation}\label{h00}
C_q(h(0,\cdot),0)=0\quad\forall\, q\in\mathbb{N}_0.
\end{equation}
Gathering \eqref{cqhinfty}, \eqref{cqhzero}, and \eqref{h00} together yields the conclusion.
\end{proof}
\begin{theorem}\label{thmfour}
Let \eqref{abeta}, $({\rm f}'_1)$, $({\rm f}_2)$, $({\rm f}_3)$, and $({\rm f}_6)$ be satisfied. Then Problem \eqref{prob} admits at least three nontrivial solutions $u_0\in{\rm int}(C_+)$, $u_1,u_2\in C^1(\overline{\Omega})$.
\end{theorem}
\begin{proof}
The same arguments adopted in the proofs of Theorems \ref{thmone}--\ref{thmtwo} give $u_0$ and $u_1$. Moreover,
\begin{equation}\label{cquzero}
C_q(\varphi,u_0)=\delta_{q,0}\mathbb{Z}\quad\forall\, q\in\mathbb{N}_0,
\end{equation}
because $u_0$ is a local minimizer for $\varphi$, while $u_1$ turns out to be a mountain pass type critical point of $\varphi$. Observe next that $\varphi\in C^2(X)$ and one has
\begin{equation}\label{2der}
\langle\varphi''(u_1)(v),w\rangle=\int_\Omega\left[\nabla v\cdot\nabla w+avw-f_t'(x,u_1)vw\right]dx+\int_{\partial\Omega}\beta vwd\sigma,\quad v,w\in X.
\end{equation}
If the Morse index of $u_1$ is zero then, by \eqref{2der},
\begin{equation}\label{finineq}
\Vert\nabla v\Vert_2^2+\int_{\partial\Omega}\beta v^2d\sigma\geq\int_\Omega\left[f_t'(x,u_1)-a\right]v^2dx\quad\mbox{in}\quad X.
\end{equation}
Two situations may now occur.\\
1)  $(f_t'(\cdot,u_1)-a)^+=0$. Given $u\in{\rm ker}(\varphi''(u_1))$, from \eqref{2der} we immediately infer
$$\Vert\nabla u\Vert_2^2+\int_{\partial\Omega}\beta u^2d\sigma\leq 0,$$
whence, on account of \eqref{abeta}, the function $u$ must be constant.\\
2) $(f_t'(\cdot,u_1)-a)^+\not\equiv 0$. Inequality \eqref{finineq} entails $\hat\lambda_1(\alpha)\geq 1$, where $\alpha:=(f_t'(\cdot,u_1)-a)$. So, due to \eqref{2der}, $\hat\lambda_1(\alpha)=1$ as soon as $u\in{\rm ker}(\varphi''(u_1))\setminus\{0\}$.\\
Consequently, in either case, ${\rm dim}({\rm ker}(\varphi''(u_1)))\leq 1$, and Corollary 6.102 of \cite{MoMoPa} yields
\begin{equation}\label{cquone}
C_q(\varphi,u_1)=\delta_{q,1}\mathbb{Z}\quad\forall\, q\in\mathbb{N}_0.
\end{equation}
 Finally, if $K(\varphi)=\{0,u_0,u_1\}$ then the Morse relation written for $t=-1$, Lemma \ref{cqphi}, \eqref{cquzero}, \eqref{cquone}, besides Lemma \ref{cqinfty} would imply
$$(-1)^{d_m}+(-1)^0+(-1)^1=0,$$
which is impossible. Thus, there exists $u_1\in K(\varphi)\setminus\{0,u_0,u_1\}$, i.e., a third nontrivial solution to \eqref{prob}. Standard regularity arguments (see the proof of Lemma \ref{auxtwo}) ensure that $u_1\in C^1(\overline{\Omega})$.
\end{proof}
\begin{example}
Let $k>\max\{n_0,2\}$ and let $b\in[\hat\lambda_k,\hat\lambda_{k+1}]$. The function $f:\Omega\times\R\to\R$ defined by, for every $(x,t)\in\Omega\times\R$,
$$f(x,t):=\left\{
\begin{array}{ll}
bt & \mbox{if } t\leq 1,\\
\hat\lambda_1t-\sqrt{t}+c\log t+d & \mbox{otherwise},\\
\end{array}
\right.$$
where $c:=b-\hat\lambda_1+2^{-1}$ and $d:=b-\hat\lambda_1+1$ satisfies all the assumptions of Theorem \ref{thmfour}.
\end{example}
\section*{Acknowledgement}
Work performed under the auspices of GNAMPA of INDAM.
\end{document}